\documentclass[10pt]{amsart}
\usepackage{epsfig,amsfonts,amsthm,amssymb,latexsym,amsmath}

\newcommand{\mymod}[3]{#1 \equiv #2 \kern -0.5em \pmod{#3}}
\newcommand{\mynotmod}[3]{#1 \not \equiv #2 \kern -0.6em \pmod{#3}}
\usepackage{framed}

\theoremstyle{plain}
\newtheorem{theorem}{Theorem}[section]
\newtheorem{corollary}[theorem]{Corollary}

\newtheorem{proposition}[theorem]{Proposition}

\theoremstyle{remark}

\theoremstyle{definition}
\newtheorem{definition}[theorem]{Definition}

\title[Third-order $k$-Jacobsthal matrix sequence]{Third-order $k$-Jacobsthal matrix sequence: Another way of demonstrating their properties}

\vspace{5pt}

\author{\scriptsize GAMALIEL CERDA-MORALES}
\date{}

\begin{document}
\maketitle

\vspace{-20pt}
\begin{center}
{\footnotesize Instituto de Matem\'aticas, Pontificia Universidad Cat\'olica de Valpara\'iso, \\
Blanco Viel 596, Cerro Bar\'on, Valpara\'iso, Chile. \\
E-mail: gamaliel.cerda.m@mail.pucv.cl
}\end{center}

\vspace{20pt}

\hrule

\begin{abstract}
Recently, Cerda-Morales \cite{Ce6} introduced commutative matrices derived from the third-order Jacobsthal matrix sequence and the third-order Jacobsthal--Lucas matrix sequence. In the present work, through the identification of certain special matrices, we can identify other forms of demonstration and also the description of commutative matrix properties for negative indices. A new generalization of this sequence is used for our purpose.
\end{abstract}

\medskip
\noindent

\medskip
\noindent
\medskip

\hrule

\section{Introduction}\label{sec:1}

The third-order Jacobsthal numbers was introduced by Cook and Bacon in his article that was published in 2013 \cite{Co}. Third-order Jacobsthal--Lucas, third-order Jacobsthal and modified third-order Jacobsthal numbers have recently been studied by Cerda-Morales \cite{Ce1,Ce4,Ce5}.

The third-order Jacobsthal and modified third-order Jacobsthal sequences $\left(J_{n}^{(3)}\right)_{n\geq 0}$ and $\left(K_{n}^{(3)}\right)_{n\geq 0}$, are defined for $n\geq 0$ by
$$J_{n+3}^{(3)}=J_{n+2}^{(3)}+J_{n+1}^{(3)}+2J_{n}^{(3)}\ \textrm{and}\ K_{n+3}^{(3)}=K_{n+2}^{(3)}+K_{n+1}^{(3)}+2K_{n}^{(3)},$$
in which $J_{0}^{(3)}=0$, $J_{1}^{(3)}=J_{2}^{(3)}=1$ and $K_{0}^{(3)}=3K_{1}^{(3)}=K_{2}^{(3)}=3$, respectively.

The Binet formulae are
\begin{equation}\label{b1}
J_{n}^{(3)}=\frac{1}{7}2^{n+1}-\frac{3+2i\sqrt{3}}{21}\omega_{1}^{n}-\frac{3-2i\sqrt{3}}{21}\omega_{2}^{n}=\frac{1}{7}\left(2^{n+1}-Z_{n}\right)
\end{equation}
and
\begin{equation}\label{b2}
K_{n}^{(3)}=2^{n}+\omega_{1}^{n}+\omega_{2}^{n}=2^{n}+Y_{n},
\end{equation}
where $$Z_{n}=\left\{ 
\begin{array}{ccc}
2 & \textrm{if} & \mymod{n}{0}{3} \\ 
-3 & \textrm{if} & \mymod{n}{1}{3} \\ 
1 & \textrm{if} & \mymod{n}{2}{3}
\end{array}
\right. \ \textrm{and}\ Y_{n}=\left\{ 
\begin{array}{ccc}
2 & \textrm{if} & \mymod{n}{0}{3} \\ 
-1 & \textrm{if} & \mynotmod{n}{0}{3}
\end{array}
\right. $$

In \cite{Ce3}, Cerda-Morales derived the following recurrence relations for the sequences $\left(J_{n}^{(3)}\right)$ and $\left(K_{n}^{(3)}\right)$. For $r\geq 0$, $n\geq 0$, 
\begin{equation}\label{m1}
J_{r(n+3)}^{(3)}=K_{r}^{(3)}J_{r(n+2)}^{(3)}-\left(2^{r}Y_{r}+1\right)J_{r(n+1)}^{(3)}+2^{r}J_{rn}^{(3)},
\end{equation}
where the initial conditions of the sequences $\left(J_{rn}^{(3)}\right)$ are $0$, $J_{n}^{(3)}$ and $J_{2r}^{(3)}$.

The Binet formulae are
\begin{equation}\label{m2}
J_{rn}^{(3)}=\frac{1}{7}\left(2^{rn+1}-Z_{rn}\right)\ \textrm{and}\ K_{rn}^{(3)}=2^{rn}+Y_{rn},
\end{equation}
respectively.

Let us consider the following definitions introduced recently in the work \cite{Ce6}. We can find several properties resulting from generalized third-order Jacobsthal sequences and the third-order Jacobsthal--Lucas sequence. On the other hand, from certain properties of matrices that generate the elements of the third-order $k$-Jacobsthal sequences and the third-order $k$-Jacobsthal--Lucas sequence we can find other more immediate and new forms of the corresponding theorems that we find in the work \cite{Ce6}. Other interesting properties about the generalized Fibonacci matrix sequences can be found at works \cite{Ci1,Ci2,Ip} and $k$-Pell matrix sequences in \cite{Re,Wa}.

\begin{definition}\label{d1}
For $k\in \mathbb{R}$ with $k>0$, the third-order $k$-Jacobsthal sequence $\left(J_{n}^{(3)}(k)\right)$ is defined by $$J_{n+3}^{(3)}(k)=(k-1)J_{n+2}^{(3)}(k)+(k-1)J_{n+1}^{(3)}(k)+kJ_{n}^{(3)}(k),$$ where $J_{0}^{(3)}(k)=0$, $J_{1}^{(3)}(k)=1$ and $J_{2}^{(3)}(k)=k-1$.
\end{definition}

\begin{definition}\label{d2}
For $k\in \mathbb{R}$, $k>0$, the third-order $k$-Jacobsthal--Lucas sequence $\left(j_{n}^{(3)}(k)\right)$ is defined by $$j_{n+3}^{(3)}(k)=(k-1)j_{n+2}^{(3)}(k)+(k-1)j_{n+1}^{(3)}(k)+kj_{n}^{(3)}(k),$$ where $j_{0}^{(3)}(k)=2$, $j_{1}^{(3)}(k)=k-1$ and $j_{2}^{(3)}(k)=k^{2}+1$.
\end{definition}

Especially, when $k=2$, then $J_{n}^{(3)}(2)=J_{n}^{(3)}$ (the $n$-th third-order Jacobsthal number) and $j_{n}^{(3)}(2)=j_{n}^{(3)}$ (the $n$-th third-order Jacobsthal--Lucas number). Next, let us look at two mathematical definitions recently introduced in \cite{Ce6} related to the matrix sequence.

\begin{definition}\label{m1}
For $k\in \mathbb{R}$, $k>0$, the third-order $k$-Jacobsthal matrix sequence $\left(\textrm{M}_{k,n}^{(3)}\right)$ is defined by $$\textrm{M}_{k,n+3}^{(3)}=(k-1)\textrm{M}_{k,n+2}^{(3)}+(k-1)\textrm{M}_{k,n+1}^{(3)}+k\textrm{M}_{k,n}^{(3)},$$ where $\textrm{M}_{k,0}^{(3)}=\left[
\begin{array}{ccc}
1& 0& 0 \\ 
0&1& 0\\ 
0 & 0&1
\end{array}
\right]$, $\textrm{M}_{k,1}^{(3)}=\left[
\begin{array}{ccc}
k-1& k-1& k \\ 
1&0& 0\\ 
0 & 1&0
\end{array}
\right]$ and \\ $\textrm{M}_{k,2}^{(3)}=\left[
\begin{array}{ccc}
k^{2}-k&k^{2}-k+1&k^{2}-k\\
k-1& k-1& k \\ 
1&0& 0
\end{array}
\right]$.
\end{definition}

\begin{definition}\label{m2}
For $k\in \mathbb{R}$, $k>0$, the third-order $k$-Jacobsthal--Lucas matrix sequence $\left(\textrm{N}_{k,n}^{(3)}\right)$ is defined by $$\textrm{N}_{k,n+3}^{(3)}=(k-1)\textrm{N}_{k,n+2}^{(3)}+(k-1)\textrm{N}_{k,n+1}^{(3)}+k\textrm{N}_{k,n}^{(3)},$$ $\textrm{N}_{k,0}^{(3)}=\left[
\begin{array}{ccc}
k-1& 2k& 2k \\ 
2&1-k& 2\\ 
\frac{2}{k} & \frac{2}{k} &-\frac{1}{k}(k^{2}+k-2)
\end{array}
\right]$, $\textrm{N}_{k,1}^{(3)}=\left[
\begin{array}{ccc}
k^{2}+1& k^{2}+1& k^{2}-k \\ 
k-1& 2k& 2k \\ 
2&1-k& 2
\end{array}
\right]$ and \\ $\textrm{N}_{k,2}^{(3)}=\left[
\begin{array}{ccc}
k^{3}+k&k^{3}-1&k^{3}+k\\
k^{2}+1& k^{2}+1& k^{2}-k \\ 
k-1& 2k& 2k \\ 
\end{array}
\right]$.
\end{definition}

Before discussing other ways of demonstrating the results addressed in the work \cite{Ce6}, we will consider the following matrix: $$\textrm{M}_{k,1}^{(3)}=\left[
\begin{array}{ccc}
k-1& k-1& k \\ 
1&0& 0\\ 
0 & 1&0
\end{array}
\right].$$ In this way, we will state the following proposition.

\begin{proposition}\label{prop1}
For any integer $n\geq 1$, we obtain 
$$
\left(\emph{\textrm{M}}_{k,1}^{(3)}\right)^{n}=\left[
\begin{array}{ccc}
J_{n+1}^{(3)}(k) & T_{n-1}^{(3)}(k) & kJ_{n}^{(3)}(k) \\ 
J_{n}^{(3)}(k) & T_{n-2}^{(3)}(k) & kJ_{n-1}^{(3)}(k)\\ 
J_{n-1}^{(3)}(k) & T_{n-3}^{(3)}(k) & kJ_{n-2}^{(3)}(k)
\end{array}
\right],
$$
where $T_{n}^{(3)}(k) =(k-1)J_{n+1}^{(3)}(k)+kJ_{n}^{(3)}(k)$.
\end{proposition}
\begin{proof}
The result holds for $n=1$: $$\left(\textrm{M}_{k,1}^{(3)}\right)^{1}=\left[
\begin{array}{ccc}
k-1& k-1& k \\ 
1&0& 0\\ 
0 & 1&0
\end{array}
\right]=\left[
\begin{array}{ccc}
J_{2}^{(3)}(k) & (k-1)J_{1}^{(3)}(k)+kJ_{0}^{(3)}(k) & kJ_{1}^{(3)}(k) \\ 
J_{1}^{(3)}(k) & (k-1)J_{0}^{(3)}(k)+kJ_{-1}^{(3)}(k)& kJ_{0}^{(3)}(k)\\ 
J_{0}^{(3)}(k) & (k-1)J_{-1}^{(3)}(k)+kJ_{-2}^{(3)}(k) & kJ_{-1}^{(3)}(k)
\end{array}
\right].$$ By mathematical induction, we assume that $$\left(\textrm{M}_{k,1}^{(3)}\right)^{n}=\left[
\begin{array}{ccc}
J_{n+1}^{(3)}(k) & T_{n-1}^{(3)}(k) & kJ_{n}^{(3)}(k) \\ 
J_{n}^{(3)}(k) & T_{n-2}^{(3)}(k) & kJ_{n-1}^{(3)}(k)\\ 
J_{n-1}^{(3)}(k) & T_{n-3}^{(3)}(k) & kJ_{n-2}^{(3)}(k)
\end{array}
\right],$$ with $T_{n}^{(3)}(k) =(k-1)J_{n+1}^{(3)}(k)+kJ_{n}^{(3)}(k)$. Next, consider the following matrix power:
\begin{align*}
\left(\textrm{M}_{k,1}^{(3)}\right)^{n+1}&=\left[
\begin{array}{ccc}
k-1& k-1& k \\ 
1&0& 0\\ 
0 & 1&0
\end{array}
\right]^{n+1}\\
&=\left[
\begin{array}{ccc}
k-1& k-1& k \\ 
1&0& 0\\ 
0 & 1&0
\end{array}
\right]^{n}\left[
\begin{array}{ccc}
k-1& k-1& k \\ 
1&0& 0\\ 
0 & 1&0
\end{array}
\right]\\
&=\left[
\begin{array}{ccc}
J_{n+1}^{(3)}(k) & T_{n-1}^{(3)}(k) & kJ_{n}^{(3)}(k) \\ 
J_{n}^{(3)}(k) & T_{n-2}^{(3)}(k) & kJ_{n-1}^{(3)}(k)\\ 
J_{n-1}^{(3)}(k) & T_{n-3}^{(3)}(k) & kJ_{n-2}^{(3)}(k)
\end{array}
\right]\left[
\begin{array}{ccc}
k-1& k-1& k \\ 
1&0& 0\\ 
0 & 1&0
\end{array}
\right]\\
&=\left[
\begin{array}{ccc}
(k-1)J_{n+1}^{(3)}(k)+T_{n-1}^{(3)}(k)  & (k-1)J_{n+1}^{(3)}(k)+kJ_{n}^{(3)}(k)  & kJ_{n+1}^{(3)}(k) \\ 
(k-1)J_{n}^{(3)}(k)+T_{n-2}^{(3)}(k)  & (k-1)J_{n}^{(3)}(k)+kJ_{n-1}^{(3)}(k)  & kJ_{n}^{(3)}(k) \\ 
(k-1)J_{n-1}^{(3)}(k)+T_{n-3}^{(3)}(k)  & (k-1)J_{n-1}^{(3)}(k)+kJ_{n-2}^{(3)}(k)  & kJ_{n-1}^{(3)}(k) 
\end{array}
\right]\\
&=\left[
\begin{array}{ccc}
J_{n+2}^{(3)}(k) & T_{n}^{(3)}(k) & kJ_{n+1}^{(3)}(k) \\ 
J_{n+1}^{(3)}(k) & T_{n-1}^{(3)}(k) & kJ_{n}^{(3)}(k)\\ 
J_{n}^{(3)}(k) & T_{n-2}^{(3)}(k) & kJ_{n-1}^{(3)}(k)
\end{array}
\right],
\end{align*}
where $T_{n}^{(3)}(k) =(k-1)J_{n+1}^{(3)}(k)+kJ_{n}^{(3)}(k)$.
\end{proof}

Next, let us consider the following matrix $$\textrm{N}_{k,0}^{(3)}=\left[
\begin{array}{ccc}
k-1& 2k& 2k \\ 
2&1-k& 2\\ 
\frac{2}{k} & \frac{2}{k} &-\frac{1}{k}(k^{2}+k-2)
\end{array}
\right].$$ 
In addition, we will take the following matrix products indicated in the expression $\textrm{N}_{k,0}^{(3)}\left(\textrm{M}_{k,1}^{(3)}\right)^{n}$. Then, we consider the following proposition.

\begin{proposition}\label{prop2}
For any integer $n\geq 1$,we obtain $$\emph{\textrm{N}}_{k,0}^{(3)}\left(\emph{\textrm{M}}_{k,1}^{(3)}\right)^{n}=\left[
\begin{array}{ccc}
j_{n+1}^{(3)}(k) & t_{n-1}^{(3)}(k) & kj_{n}^{(3)}(k) \\ 
j_{n}^{(3)}(k) & t_{n-2}^{(3)}(k) & kj_{n-1}^{(3)}(k)\\ 
j_{n-1}^{(3)}(k) & t_{n-3}^{(3)}(k) & kj_{n-2}^{(3)}(k)
\end{array}
\right],$$
where $t_{n}^{(3)}(k) =(k-1)j_{n+1}^{(3)}(k)+kj_{n}^{(3)}(k)$.
\end{proposition}
\begin{proof}
Similar to the previous theorem, the result holds for $n=1$ and by mathematical induction, it is enough to verify
\begin{align*}
\textrm{N}_{k,0}^{(3)}\left(\textrm{M}_{k,1}^{(3)}\right)^{n+1}&=\textrm{N}_{k,0}^{(3)}\left(\textrm{M}_{k,1}^{(3)}\right)^{n}\textrm{M}_{k,1}^{(3)}\\
&=\left[
\begin{array}{ccc}
j_{n+1}^{(3)}(k) & t_{n-1}^{(3)}(k) & kj_{n}^{(3)}(k) \\ 
j_{n}^{(3)}(k) & t_{n-2}^{(3)}(k) & kj_{n-1}^{(3)}(k)\\ 
j_{n-1}^{(3)}(k) & t_{n-3}^{(3)}(k) & kj_{n-2}^{(3)}(k)
\end{array}
\right]\left[
\begin{array}{ccc}
k-1& k-1& k \\ 
1&0& 0\\ 
0 & 1&0
\end{array}
\right]\\
&=\left[
\begin{array}{ccc}
(k-1)j_{n+1}^{(3)}(k)+t_{n-1}^{(3)}(k)  & (k-1)j_{n+1}^{(3)}(k)+kj_{n}^{(3)}(k)  & kj_{n+1}^{(3)}(k) \\ 
(k-1)j_{n}^{(3)}(k)+t_{n-2}^{(3)}(k)  & (k-1)j_{n}^{(3)}(k)+kj_{n-1}^{(3)}(k)  & kj_{n}^{(3)}(k) \\ 
(k-1)j_{n-1}^{(3)}(k)+t_{n-3}^{(3)}(k)  & (k-1)j_{n-1}^{(3)}(k)+kj_{n-2}^{(3)}(k)  & kj_{n-1}^{(3)}(k) 
\end{array}
\right]\\
&=\left[
\begin{array}{ccc}
j_{n+2}^{(3)}(k) & t_{n}^{(3)}(k) & kj_{n+1}^{(3)}(k) \\ 
j_{n+1}^{(3)}(k) & t_{n-1}^{(3)}(k) & kj_{n}^{(3)}(k)\\ 
j_{n}^{(3)}(k) & t_{n-2}^{(3)}(k) & kj_{n-1}^{(3)}(k)
\end{array}
\right],
\end{align*}
for any integer $n\geq 1$ and $t_{n}^{(3)}(k) =(k-1)j_{n+1}^{(3)}(k)+kj_{n}^{(3)}(k)$.
\end{proof}

In addition, we can also verify the behavior of the following determinants indicated by 
\begin{align*}
\det\left[\left(\textrm{M}_{k,1}^{(3)}\right)^{n}\right]&=\det\left[
\begin{array}{ccc}
J_{n+1}^{(3)}(k) & T_{n-1}^{(3)}(k) & kJ_{n}^{(3)}(k) \\ 
J_{n}^{(3)}(k) & T_{n-2}^{(3)}(k) & kJ_{n-1}^{(3)}(k)\\ 
J_{n-1}^{(3)}(k) & T_{n-3}^{(3)}(k) & kJ_{n-2}^{(3)}(k)
\end{array}
\right],\\
\det\left[\textrm{N}_{k,0}^{(3)}\left(\textrm{M}_{k,1}^{(3)}\right)^{n}\right]&=\det\left[
\begin{array}{ccc}
j_{n+1}^{(3)}(k) & t_{n-1}^{(3)}(k) & kj_{n}^{(3)}(k) \\ 
j_{n}^{(3)}(k) & t_{n-2}^{(3)}(k) & kj_{n-1}^{(3)}(k)\\ 
j_{n-1}^{(3)}(k) & t_{n-3}^{(3)}(k) & kj_{n-2}^{(3)}(k)
\end{array}
\right],
\end{align*}
where $T_{n}^{(3)}(k)$ and $t_{n}^{(3)}(k)$ as in propositions (\ref{prop1}) and (\ref{prop2}).

\begin{corollary}
For any integer $n\geq 1$, we obtain:
\begin{equation}
\det\left[\left(\emph{\textrm{M}}_{k,1}^{(3)}\right)^{n}\right]=k^{n},
\end{equation}
\begin{equation}
\det\left[\emph{\textrm{N}}_{k,0}^{(3)}\left(\emph{\textrm{M}}_{k,1}^{(3)}\right)^{n}\right]=(k+1)^{2}(k^{2}+k+2)k^{n-1}.
\end{equation}
\end{corollary}
\begin{proof}
We note that $\det\left[\textrm{M}_{k,1}^{(3)}\right]=k$. By mathematical induction, let us write:
\begin{align*}
\det\left[\left(\textrm{M}_{k,1}^{(3)}\right)^{n+1}\right]&=\det\left[\left(\textrm{M}_{k,1}^{(3)}\right)^{n}\cdot \textrm{M}_{k,1}^{(3)}\right]\\
& =\det\left[\left(\textrm{M}_{k,1}^{(3)}\right)^{n}\right]\cdot \det\left[\textrm{M}_{k,1}^{(3)}\right]\\
& =k^{n}\cdot k=k^{n+1}.
\end{align*}
Similarly, let us admit that $\det\left[\textrm{N}_{k,0}^{(3)}\left(\textrm{M}_{k,1}^{(3)}\right)^{n}\right]=(k+1)^{2}(k^{2}+k+2)k^{n-1}$. Thus, we can see that 
\begin{align*}
\det\left[\textrm{N}_{k,0}^{(3)}\left(\textrm{M}_{k,1}^{(3)}\right)^{n+1}\right]&=\det\left[\textrm{N}_{k,0}^{(3)}\left(\textrm{M}_{k,1}^{(3)}\right)^{n}\cdot \textrm{M}_{k,1}^{(3)}\right]\\
& =\det\left[\textrm{N}_{k,0}^{(3)}\left(\textrm{M}_{k,1}^{(3)}\right)^{n}\right]\cdot \det\left[\textrm{M}_{k,1}^{(3)}\right]\\
& =(k+1)^{2}(k^{2}+k+2)k^{n-1}\cdot k\\
& = (k+1)^{2}(k^{2}+k+2)k^{n}.
\end{align*}
\end{proof}

\section{Other demonstrations for the commutative properties}
For any integer $n\geq 1$, we define
$$\textrm{J}_{n}(k)=\left(\textrm{M}_{k,1}^{(3)}\right)^{n}$$ and $$\textrm{j}_{n}(k)=\textrm{N}_{k,0}^{(3)}\left(\textrm{M}_{k,1}^{(3)}\right)^{n},$$ are two matrices of order 3, with real coefficients.

In this section, we will discuss other simpler and more immediate ways to demonstrate the matrix properties of the matrices previously defined.

\begin{theorem}
For $m,\ n \geq 1$, the following results hold:
\begin{equation}
\emph{\textrm{J}}_{m+n}(k)=\emph{\textrm{J}}_{m}(k)\cdot \emph{\textrm{J}}_{n}(k)=\emph{\textrm{J}}_{n}(k) \cdot \emph{\textrm{J}}_{m}(k),
\end{equation}
\begin{equation}
\emph{\textrm{j}}_{m}(k)\cdot \emph{\textrm{j}}_{n}(k)=\emph{\textrm{j}}_{n}(k) \cdot \emph{\textrm{j}}_{m}(k),
\end{equation}
\begin{equation}
\emph{\textrm{J}}_{m}(k)\cdot \emph{\textrm{j}}_{n}(k)=\emph{\textrm{j}}_{n}(k) \cdot \emph{\textrm{J}}_{m}(k),
\end{equation}
\end{theorem}
\begin{proof}
For the first item, we see that the following commutative properties occur
\begin{align*}
\textrm{J}_{m+n}(k)=\left(\textrm{M}_{k,1}^{(3)}\right)^{m+n}&= \left(\textrm{M}_{k,1}^{(3)}\right)^{m}\cdot \left(\textrm{M}_{k,1}^{(3)}\right)^{n}\\
& = \left(\textrm{M}_{k,1}^{(3)}\right)^{n}\cdot \left(\textrm{M}_{k,1}^{(3)}\right)^{m}\\
&  = \textrm{J}_{n+m}(k).
\end{align*}
The second equation of this theorem, let us also see that 
 \begin{align*}
\textrm{j}_{m}(k)\cdot \textrm{j}_{n}(k)&=\textrm{N}_{k,0}^{(3)}\left(\textrm{M}_{k,1}^{(3)}\right)^{m} \cdot \textrm{N}_{k,0}^{(3)}\left(\textrm{M}_{k,1}^{(3)}\right)^{n}\\
& = \textrm{N}_{k,0}^{(3)}\left(\textrm{M}_{k,1}^{(3)}\right)^{n} \cdot \textrm{N}_{k,0}^{(3)}\left(\textrm{M}_{k,1}^{(3)}\right)^{m}\\
&  = \textrm{j}_{n}(k)\cdot \textrm{j}_{m}(k)
\end{align*}
and we record the commutative property of $\textrm{N}_{k,0}^{(3)}\cdot \textrm{M}_{k,1}^{(3)}=\textrm{M}_{k,1}^{(3)}\cdot \textrm{N}_{k,0}^{(3)}$. To conclude, let us easily see that $$\textrm{J}_{m}(k)\cdot \textrm{j}_{n}(k)=\left(\textrm{M}_{k,1}^{(3)}\right)^{m} \cdot \textrm{N}_{k,0}^{(3)}\left(\textrm{M}_{k,1}^{(3)}\right)^{n}= \textrm{N}_{k,0}^{(3)}\left(\textrm{M}_{k,1}^{(3)}\right)^{n}\cdot  \left(\textrm{M}_{k,1}^{(3)}\right)^{m}=\textrm{j}_{n}(k) \cdot \textrm{J}_{m}(k).$$
\end{proof}

\begin{proposition}\label{pr}
For $n \geq 2$, the following results holds:
\begin{equation}\label{eq:1}
\emph{\textrm{j}}_{n}(k)=(k-1)\emph{\textrm{J}}_{n}(k)+2k\emph{\textrm{J}}_{n-1}(k)+2k\emph{\textrm{J}}_{n-2}(k),
\end{equation}
\begin{equation}\label{eq:2}
\emph{\textrm{j}}_{n}(k)=2\emph{\textrm{J}}_{n+1}(k)+(1-k)\emph{\textrm{J}}_{n}(k)+2\emph{\textrm{J}}_{n-1}(k).
\end{equation}
\end{proposition}
\begin{proof}
(\ref{eq:1}): For $n \geq 2$, we can get
\begin{align*}
(k-1)\emph{\textrm{J}}_{n}(k)+&2k\emph{\textrm{J}}_{n-1}(k)+2k\emph{\textrm{J}}_{n-2}(k)\\
& =(k-1)\cdot \left(\textrm{M}_{k,1}^{(3)}\right)^{n}+2k\cdot \left(\textrm{M}_{k,1}^{(3)}\right)^{n-1}+2k\cdot \left(\textrm{M}_{k,1}^{(3)}\right)^{n-2}\\
& =\left(\textrm{M}_{k,1}^{(3)}\right)^{n} \left[(k-1)\cdot \textrm{I}+2k\cdot \left(\textrm{M}_{k,1}^{(3)}\right)^{-1}+2k\cdot \left(\textrm{M}_{k,1}^{(3)}\right)^{-2}\right].
\end{align*}
On the other hand, we can determine that
\begin{align*}
(k-1)\cdot \textrm{I}+&2k\cdot \left(\textrm{M}_{k,1}^{(3)}\right)^{-1}+2k\cdot \left(\textrm{M}_{k,1}^{(3)}\right)^{-2}\\
& =(k-1)\left[
\begin{array}{ccc}
1& 0& 0 \\ 
0& 1 & 0\\ 
0& 0& 1
\end{array}
\right]+2k\left[
\begin{array}{ccc}
0& 1& 0 \\ 
0& 0 & 1\\ 
\frac{1}{k}& \frac{1-k}{k} & \frac{1-k}{k}
\end{array}
\right]\\
& \ +2k\left[
\begin{array}{ccc}
0& 0 & 1\\ 
\frac{1}{k}& \frac{1-k}{k} & \frac{1-k}{k}\\
\frac{1-k}{k^{2}}& \frac{k^{2}-k+1}{k^{2}} & \frac{1-k}{k^{2}}
\end{array}
\right]\\
& =\left[
\begin{array}{ccc}
k-1& 2k& 2k \\ 
2&1-k& 2\\ 
\frac{2}{k} & \frac{2}{k} &-\frac{1}{k}(k^{2}+k-2)
\end{array}
\right]= \textrm{N}_{k,0}^{(3)}.
\end{align*}
Finally, we determine equality $$(k-1)\emph{\textrm{J}}_{n}(k)+2k\emph{\textrm{J}}_{n-1}(k)+2k\emph{\textrm{J}}_{n-2}(k)=\left(\textrm{M}_{k,1}^{(3)}\right)^{n} \cdot  \textrm{N}_{k,0}^{(3)}=\textrm{N}_{k,0}^{(3)}\cdot \left(\textrm{M}_{k,1}^{(3)}\right)^{n} =\textrm{j}_{n}(k).$$

(\ref{eq:2}) The proof is similar to the proof of (\ref{eq:1}).
\end{proof}

\begin{theorem}
For $n \geq 1$, the following results holds:
\begin{equation}\label{a:1}
\left(\emph{\textrm{j}}_{n+1}(k)\right)^{2}=\left(\emph{\textrm{j}}_{1}(k)\right)^{2}\cdot \emph{\textrm{J}}_{2n}(k),
\end{equation}
\begin{equation}\label{a:2}
\emph{\textrm{j}}_{2n+1}(k)=\emph{\textrm{J}}_{n}(k)\cdot \emph{\textrm{j}}_{n+1}(k).
\end{equation}
\end{theorem}
\begin{proof}
(\ref{a:1}) In this case, we can determine directly from the definition that:
\begin{align*}
\left(\textrm{j}_{n+1}(k)\right)^{2}&=\textrm{j}_{n+1}(k) \cdot \textrm{j}_{n+1}(k)\\
&=\textrm{N}_{k,0}^{(3)}\cdot \left(\textrm{M}_{k,1}^{(3)}\right)^{n+1}\cdot \textrm{N}_{k,0}^{(3)}\cdot \left(\textrm{M}_{k,1}^{(3)}\right)^{n+1}\\
&=\textrm{N}_{k,0}^{(3)}\textrm{M}_{k,1}^{(3)}\cdot \left(\textrm{M}_{k,1}^{(3)}\right)^{n}\cdot \textrm{N}_{k,0}^{(3)}\textrm{M}_{k,1}^{(3)}\cdot \left(\textrm{M}_{k,1}^{(3)}\right)^{n}\\
&=\textrm{j}_{1}(k)\cdot \left(\textrm{M}_{k,1}^{(3)}\right)^{n}\cdot \textrm{j}_{1}(k)\cdot \left(\textrm{M}_{k,1}^{(3)}\right)^{n}\\
&=\left(\textrm{j}_{1}(k)\right)^{2}\cdot \left(\textrm{M}_{k,1}^{(3)}\right)^{2n}\\
&=\left(\textrm{j}_{1}(k)\right)^{2}\cdot \textrm{J}_{2n}(k).
\end{align*}

(\ref{a:2}) Similarly, we see that
\begin{align*}
\textrm{j}_{2n+1}(k)&=\textrm{N}_{k,0}^{(3)}\left(\textrm{M}_{k,1}^{(3)}\right)^{2n+1}\\
&=\textrm{N}_{k,0}^{(3)}\left(\textrm{M}_{k,1}^{(3)}\right)^{n}\cdot \left(\textrm{M}_{k,1}^{(3)}\right)^{n+1}\\
&=\left(\textrm{M}_{k,1}^{(3)}\right)^{n}\cdot \textrm{N}_{k,0}^{(3)}\left(\textrm{M}_{k,1}^{(3)}\right)^{n+1}\\
&=\textrm{J}_{n}(k)\cdot \textrm{j}_{n+1}(k).
\end{align*}
\end{proof}
\begin{theorem}
For any integers $m, n \geq 1$, we obtain $$\emph{\textrm{j}}_{m+n}(k)=\emph{\textrm{j}}_{m}(k)\cdot \emph{\textrm{J}}_{n}(k)=\emph{\textrm{J}}_{m}(k)\cdot \emph{\textrm{j}}_{n}(k).$$
\end{theorem}
\begin{proof}
Let us consider that $$\textrm{j}_{m+n}(k)=\textrm{N}_{k,0}^{(3)}\left(\textrm{M}_{k,1}^{(3)}\right)^{m+n}=\textrm{N}_{k,0}^{(3)}\left(\textrm{M}_{k,1}^{(3)}\right)^{n}\cdot \left(\textrm{M}_{k,1}^{(3)}\right)^{m}=\textrm{j}_{n}(k)\textrm{J}_{m}(k).$$ In addition, we can write immediately that $$\textrm{j}_{m+n}(k)=\textrm{j}_{0}(k)\textrm{J}_{m+n}(k)=\textrm{j}_{0}(k)\textrm{J}_{n}(k)\cdot \textrm{J}_{m}(k)=\textrm{J}_{n}(k)\cdot \textrm{j}_{0}(k)\textrm{J}_{m}(k)=\textrm{J}_{n}(k)\cdot \textrm{j}_{m}(k)$$ since, we know the commutativity of the matrix product $\textrm{J}_{1}(k)\cdot \textrm{j}_{0}(k)=\textrm{j}_{0}(k)\cdot \textrm{J}_{1}(k)$. The proof is completed.
\end{proof}

\section{Matrix sequence properties for negative indices}
Now, we will develop the study of certain properties determined by the following inverse third-order $k$-Jacobsthal matrix indicated by $\left(\textrm{M}_{k,1}^{(3)}\right)^{-n}$. We can immediately determine some particular cases
$$
\left(\textrm{M}_{k,1}^{(3)}\right)^{-1}=\left[
\begin{array}{ccc}
0& 1& 0 \\ 
0& 0 & 1\\ 
\frac{1}{k}& \frac{1-k}{k}& \frac{1-k}{k}
\end{array}
\right],$$ and 
$$\left(\textrm{M}_{k,1}^{(3)}\right)^{-2}=\left[
\begin{array}{ccc}
0& 0 & 1\\ 
\frac{1}{k}& \frac{1-k}{k}& \frac{1-k}{k}\\
\frac{1-k}{k^{2}}& \frac{1-k+k^{2}}{k^{2}}& \frac{1-k}{k^{2}}
\end{array}
\right].$$

We have observed that the elements of the type $J_{-n}^{(3)}(k)$, for a positive integer $n\geq 0$ can be determined directly from the recurrence relation indicated by 
$$J_{-n}^{(3)}(k)=\frac{1-k}{k}J_{-(n-1)}^{(3)}(k)+\frac{1-k}{k}J_{-(n-2)}^{(3)}(k)+\frac{1}{k}J_{-(n-3)}^{(3)}(k).$$ From these preliminary examples, we will state the following theorem.

\begin{theorem}
For any integer $n\geq 1$, we obtain
\begin{align*}
\emph{\textrm{J}}_{-n}(k)&=\left[
\begin{array}{ccc}
J_{-(n-1)}^{(3)}(k) & T_{-(n+1)}^{(3)}(k) & kJ_{-n}^{(3)}(k) \\ 
J_{-n}^{(3)}(k) & T_{-(n+2)}^{(3)}(k) & kJ_{-(n+1)}^{(3)}(k)\\ 
J_{-(n+1)}^{(3)}(k) & T_{-(n+3)}^{(3)}(k) & kJ_{-(n+2)}^{(3)}(k)
\end{array}
\right],\\
\emph{\textrm{J}}_{-n}(k)&=\left(\textrm{M}_{k,1}^{(3)}\right)^{-n}=\left[\left(\textrm{M}_{k,1}^{(3)}\right)^{-1}\right]^{n}=\left[\left(\textrm{M}_{k,1}^{(3)}\right)^{n}\right]^{-1},
\end{align*}
where $T_{-n}^{(3)}(k) =(k-1)T_{-(n-1)}^{(3)}(k)+kJ_{-n}^{(3)}(k)$.
\end{theorem}
\begin{proof}
The result holds for $n=1$: 
\begin{align*}
\textrm{J}_{-1}(k)&=\left[
\begin{array}{ccc}
0& 1& 0 \\ 
0& 0 & 1\\ 
\frac{1}{k}& \frac{1-k}{k}& \frac{1-k}{k}
\end{array}
\right]=\left[
\begin{array}{ccc}
J_{0}^{(3)}(k) & (k-1)J_{-1}^{(3)}(k)+kJ_{-2}^{(3)}(k) & kJ_{-1}^{(3)}(k) \\ 
J_{-1}^{(3)}(k) & (k-1)J_{-2}^{(3)}(k)+kJ_{-3}^{(3)}(k)& kJ_{-2}^{(3)}(k)\\ 
J_{-2}^{(3)}(k) & (k-1)J_{-3}^{(3)}(k)+kJ_{-4}^{(3)}(k) & kJ_{-3}^{(3)}(k)
\end{array}
\right].
\end{align*}
By mathematical induction, we assume that $$\textrm{J}_{-n}(k)=\left[
\begin{array}{ccc}
J_{-(n-1)}^{(3)}(k) & T_{-(n+1)}^{(3)}(k) & kJ_{-n}^{(3)}(k) \\ 
J_{-n}^{(3)}(k) & T_{-(n+2)}^{(3)}(k) & kJ_{-(n+1)}^{(3)}(k)\\ 
J_{-(n+1)}^{(3)}(k) & T_{-(n+3)}^{(3)}(k) & kJ_{-(n+2)}^{(3)}(k)
\end{array}
\right],$$ with $T_{-n}^{(3)}(k) =(k-1)J_{-(n-1)}^{(3)}(k)+kJ_{-n}^{(3)}(k)$. Next, consider the following matrix power:
\begin{align*}
\textrm{J}_{-(n+1)}(k)&=\left[
\begin{array}{ccc}
0& 1& 0 \\ 
0& 0 & 1\\ 
\frac{1}{k}& \frac{1-k}{k}& \frac{1-k}{k}
\end{array}
\right]^{n+1}\\
&=\left[
\begin{array}{ccc}
0& 1& 0 \\ 
0& 0 & 1\\ 
\frac{1}{k}& \frac{1-k}{k}& \frac{1-k}{k}
\end{array}
\right]^{n}\left[
\begin{array}{ccc}
0& 1& 0 \\ 
0& 0 & 1\\ 
\frac{1}{k}& \frac{1-k}{k}& \frac{1-k}{k}
\end{array}
\right].
\end{align*}
Then, we obtain
\begin{align*}
\textrm{J}_{-(n+1)}(k)&=\left[
\begin{array}{ccc}
J_{-(n-1)}^{(3)}(k) & T_{-(n+1)}^{(3)}(k) & kJ_{-n}^{(3)}(k) \\ 
J_{-n}^{(3)}(k) & T_{-(n+2)}^{(3)}(k) & kJ_{-(n+1)}^{(3)}(k)\\ 
J_{-(n+1)}^{(3)}(k) & T_{-(n+3)}^{(3)}(k) & kJ_{-(n+2)}^{(3)}(k)
\end{array}
\right]\left[
\begin{array}{ccc}
0& 1& 0 \\ 
0& 0 & 1\\ 
\frac{1}{k}& \frac{1-k}{k}& \frac{1-k}{k}
\end{array}
\right]\\
&=\left[
\begin{array}{ccc}
J_{-n}^{(3)}(k) & J_{-(n-1)}^{(3)}(k)+(1-k)J_{-n}^{(3)}(k)  & T_{-(n+1)}^{(3)}(k)+(1-k)J_{-n}^{(3)}(k) \\ 
J_{-(n+1)}^{(3)}(k) & J_{-n}^{(3)}(k)+(1-k)J_{-(n+1)}^{(3)}(k)  & T_{-(n+2)}^{(3)}(k)+(1-k)J_{-(n+1)}^{(3)}(k) \\ 
J_{-(n+2)}^{(3)}(k) & J_{-(n+1)}^{(3)}(k)+(1-k)J_{-(n+2)}^{(3)}(k)  & T_{-(n+3)}^{(3)}(k)+(1-k)J_{-(n+2)}^{(3)}(k)
\end{array}
\right]\\
&=\left[
\begin{array}{ccc}
J_{-n}^{(3)}(k) & T_{-(n+2)}^{(3)}(k) & kJ_{-(n+1)}^{(3)}(k) \\ 
J_{-(n+1)}^{(3)}(k) & T_{-(n+3)}^{(3)}(k) & kJ_{-(n+2)}^{(3)}(k)\\ 
J_{-(n+2)}^{(3)}(k) & T_{-(n+4)}^{(3)}(k) & kJ_{-(n+3)}^{(3)}(k)
\end{array}
\right],
\end{align*}
and using the following relation $$J_{-n}^{(3)}(k)=(k-1)J_{-(n+1)}^{(3)}(k)+(k-1)J_{-(n+2)}^{(3)}(k)+kJ_{-(n+3)}^{(3)}(k),$$ note that $T_{-(n+2)}^{(3)}(k) =J_{-(n-1)}^{(3)}(k)+(1-k)J_{-n}^{(3)}(k)$. In this way, we will write 
\begin{align*}
\textrm{J}_{-n}(k)&=\left[
\begin{array}{ccc}
0& 1& 0 \\ 
0& 0 & 1\\ 
\frac{1}{k}& \frac{1-k}{k}& \frac{1-k}{k}
\end{array}
\right]^{n}=\left[\left(\textrm{M}_{k,1}^{(3)}\right)^{-1}\right]^{n}.
\end{align*}
\end{proof}

In the following theorem, we will reduce the Binet formula corresponding to the terms of negative indices and not discussed in \cite{Ce6}.

\begin{theorem}
For any integer $n\geq 1$, we obtain 
\begin{equation}
J_{-n}^{(3)}(k) =\frac{1}{k^{2}+k+1}\left[k\left(\frac{1}{k}\right)^{n}+\frac{B\omega_{1}^{n}-A\omega_{2}^{n}}{\omega_{1}-\omega_{2}}\right],
\end{equation}
where $A=\omega_{1}k-1$, $B=\omega_{2}k-1$ and $\omega_{1},\omega_{2}$ are roots of the polynomial $x^{2}+x+1$.
\end{theorem}
\begin{proof}
We know that $\omega_{1}\omega_{2}=1$. In this way, we will make the following substitutions
\begin{align*}
J_{-n}^{(3)}(k)& =\frac{1}{k^{2}+k+1}\left[k^{-n+1}-\frac{A\omega_{1}^{-n}-B\omega_{2}^{-n}}{\omega_{1}-\omega_{2}}\right]\\
&=\frac{1}{k^{2}+k+1}\left[k\cdot k^{-n}-\frac{\frac{A}{\omega_{1}^{n}}-\frac{B}{\omega_{2}^{n}}}{\omega_{1}-\omega_{2}}\right]\\
&=\frac{1}{k^{2}+k+1}\left[k\left(\frac{1}{k}\right)^{n}+\frac{B\omega_{1}^{n}-A\omega_{2}^{n}}{\omega_{1}-\omega_{2}}\right],
\end{align*}
using the Binet formula $J_{n}^{(3)}(k) =\frac{1}{k^{2}+k+1}\left[k^{n+1}-\frac{A\omega_{1}^{n}-B\omega_{2}^{n}}{\omega_{1}-\omega_{2}}\right]$.
\end{proof}

Let us consider the following equation $\textrm{j}_{n}(k)=2\textrm{J}_{n+1}(k)+(1-k)\textrm{J}_{n}(k)+2\textrm{J}_{n-1}(k)$ from Proposition \ref{pr}. We will reduce the corresponding identity to the terms of negative indices.

\begin{theorem}
For any integer $n\geq 1$, we obtain that
$$j_{-n}^{(3)}(k) =2J_{-(n-1)}^{(3)}(k)+(1-k)J_{-n}^{(3)}(k) +2J_{-(n+1)}^{(3)}(k).$$
\end{theorem}
\begin{proof}
We can observe that
\begin{align*}
(k^{2}+k+1) &\cdot \left[2J_{-(n-1)}^{(3)}(k)+(1-k)J_{-n}^{(3)}(k) +2J_{-(n+1)}^{(3)}(k)\right]\\
&=2k\left(\frac{1}{k}\right)^{n-1}+2\frac{B\omega_{1}^{n-1}-A\omega_{2}^{n-1}}{\omega_{1}-\omega_{2}}\\
&\ \ +(1-k)k\left(\frac{1}{k}\right)^{n}+(1-k)\frac{B\omega_{1}^{n}-A\omega_{2}^{n}}{\omega_{1}-\omega_{2}}\\
&\ \ +2k\left(\frac{1}{k}\right)^{n+1}+2\frac{B\omega_{1}^{n+1}-A\omega_{2}^{n+1}}{\omega_{1}-\omega_{2}}\\
&=(k^{2}+k+2)\left(\frac{1}{k}\right)^{n}+\left(\frac{1}{\omega_{1}}+(1-k)+2\omega_{1}\right)\frac{B\omega_{1}^{n}-A\omega_{2}^{n}}{\omega_{1}-\omega_{2}}.
\end{align*}
On the other hand, we can directly verify that $\omega_{2}=\frac{1}{\omega_{1}}$ and $\omega_{1}+\omega_{2}=-1$. Finally, we deduce $\frac{1}{\omega_{1}}+(1-k)+2\omega_{1}=-(k+1)$ and $$2J_{-(n-1)}^{(3)}(k)+(1-k)J_{-n}^{(3)}(k) +2J_{-(n+1)}^{(3)}(k)=j_{-n}^{(3)}(k).$$
\end{proof}

Now, we see the following theorem that allows determining the generating matrices for the family of matrices $\{\textrm{j}_{-n}(k)\}_{n\in \mathbb{N}}$, which we have preliminarily defined by $$\textrm{j}_{-n}(k)=\left[
\begin{array}{ccc}
j_{-(n-1)}^{(3)}(k) & t_{-(n+1)}^{(3)}(k) & kj_{-n}^{(3)}(k) \\ 
j_{-n}^{(3)}(k) & t_{-(n+2)}^{(3)}(k) & kj_{-(n+1)}^{(3)}(k)\\ 
j_{-(n+1)}^{(3)}(k) & t_{-(n+3)}^{(3)}(k) & kj_{-(n+2)}^{(3)}(k)
\end{array}
\right],$$ where $t_{-n}^{(3)}(k) =(k-1)j_{-(n-1)}^{(3)}(k)+kj_{-n}^{(3)}(k)$.

\begin{theorem}
For any integer $n\geq 1$, we obtain that
\begin{equation}\label{b1}
\emph{\textrm{j}}_{-n}(k)=\left(\emph{\textrm{J}}_{1}(k)\right)^{-n}\cdot \emph{\textrm{j}}_{0}(k)=\emph{\textrm{j}}_{0}(k)\cdot \left(\emph{\textrm{J}}_{1}(k)\right)^{-n},
\end{equation}
\begin{equation}\label{b2}
\left(\emph{\textrm{j}}_{n}(k)\right)^{-1}=\left(\textrm{j}_{0}(k)\right)^{-1}\cdot \textrm{j}_{-n}(k)\cdot \left(\textrm{j}_{0}(k)\right)^{-1}.
\end{equation}
\end{theorem}
\begin{proof}
(\ref{b1}: From the previous theorem we will consider
\begin{align*}
\textrm{j}_{-n}(k)&=\left[
\begin{array}{ccc}
j_{-(n-1)}^{(3)}(k) & t_{-(n+1)}^{(3)}(k) & kj_{-n}^{(3)}(k) \\ 
j_{-n}^{(3)}(k) & t_{-(n+2)}^{(3)}(k) & kj_{-(n+1)}^{(3)}(k)\\ 
j_{-(n+1)}^{(3)}(k) & t_{-(n+3)}^{(3)}(k) & kj_{-(n+2)}^{(3)}(k)
\end{array}
\right]\\
&=2\left[
\begin{array}{ccc}
J_{-(n-2)}^{(3)}(k) & T_{-n}^{(3)}(k) & kJ_{-(n-1)}^{(3)}(k) \\ 
J_{-(n-1)}^{(3)}(k) & T_{-(n+1)}^{(3)}(k) & kJ_{-n}^{(3)}(k)\\ 
J_{-n}^{(3)}(k) & T_{-(n+2)}^{(3)}(k) & kJ_{-(n+1)}^{(3)}(k)
\end{array}
\right]\\
& \ \ + (1-k)\left[
\begin{array}{ccc}
J_{-(n-1)}^{(3)}(k) & T_{-(n+1)}^{(3)}(k) & kJ_{-n}^{(3)}(k) \\ 
J_{-n}^{(3)}(k) & T_{-(n+2)}^{(3)}(k) & kJ_{-(n+1)}^{(3)}(k)\\ 
J_{-(n+1)}^{(3)}(k) & T_{-(n+3)}^{(3)}(k) & kJ_{-(n+2)}^{(3)}(k)
\end{array}
\right]\\
& \ \ + 2 \left[
\begin{array}{ccc}
J_{-n}^{(3)}(k) & T_{-(n+2)}^{(3)}(k) & kJ_{-(n+1)}^{(3)}(k) \\ 
J_{-(n+1)}^{(3)}(k) & T_{-(n+3)}^{(3)}(k) & kJ_{-(n+2)}^{(3)}(k)\\ 
J_{-(n+2)}^{(3)}(k) & T_{-(n+4)}^{(3)}(k) & kJ_{-(n+3)}^{(3)}(k)
\end{array}
\right]\\
&=2\textrm{J}_{-(n-1)}(k)+(1-k)\textrm{J}_{-n}(k)+2\textrm{J}_{-(n+1)}(k).
\end{align*}
On the other hand, we know that $\textrm{J}_{-n}(k)=\left(\textrm{M}_{k,1}^{(3)}\right)^{-n}$. In this way, we will make the corresponding substitutions to determine that
\begin{align*}
\textrm{j}_{-n}(k)&=\left(\textrm{M}_{k,1}^{(3)}\right)^{-n}\cdot \left[2\textrm{M}_{k,1}^{(3)}+(1-k)\textrm{I}+2\left(\textrm{M}_{k,1}^{(3)}\right)^{-1}\right]\\
& =\left(\textrm{M}_{k,1}^{(3)}\right)^{-n}\cdot \left[
\begin{array}{ccc}
k-1& 2k& 2k \\ 
2&1-k& 2\\ 
\frac{2}{k} & \frac{2}{k} &-\frac{1}{k}(k^{2}+k-2)
\end{array}
\right]\\
&=\left(\textrm{M}_{k,1}^{(3)}\right)^{-n}\cdot \textrm{N}_{k,0}^{(3)}\\
&=\textrm{N}_{k,0}^{(3)} \cdot \left(\textrm{M}_{k,1}^{(3)}\right)^{-n}=\textrm{j}_{0}(k)\cdot \textrm{J}_{-n}(k).
\end{align*}

(\ref{b2}): From Eq. (\ref{b1}) and $\textrm{j}_{n}(k)=\textrm{N}_{k,0}^{(3)} \cdot \left(\textrm{M}_{k,1}^{(3)}\right)^{n}$, we know that 
\begin{align*}
\left(\textrm{j}_{n}(k)\right)^{-1}&=\left(\textrm{M}_{k,1}^{(3)}\right)^{-n}\cdot \left(\textrm{N}_{k,0}^{(3)}\right)^{-1}\\
&=\left(\textrm{N}_{k,0}^{(3)}\right)^{-1}\cdot \textrm{N}_{k,0}^{(3)}\left(\textrm{M}_{k,1}^{(3)}\right)^{-n}\cdot \left(\textrm{N}_{k,0}^{(3)}\right)^{-1}\\
&=\left(\textrm{j}_{0}(k)\right)^{-1}\cdot \textrm{j}_{-n}(k)\cdot \left(\textrm{j}_{0}(k)\right)^{-1}.
\end{align*}
The proof is completed.
\end{proof}

\medskip

\end{document}